\documentclass[reqno,a4paper]{amsart}

\usepackage[latin1]{inputenc}
\usepackage[english]{babel}
\usepackage{eucal,amsfonts,amssymb,amsmath,amsthm,epsfig,mathrsfs}
\usepackage{dsfont,cancel,soul}
\usepackage{color}
\usepackage{amscd,amsxtra}
\usepackage{enumerate}
\usepackage{latexsym}
\usepackage{bm}

\allowdisplaybreaks

\makeatletter
\@namedef{subjclassname@2020}{%
  \textup{2020} Mathematics Subject Classification}
\makeatother
 \makeatletter \@addtoreset{equation}{section}

\makeatother \makeatletter

\newtheorem{theorem}{Theorem}[section]
{\rm}
{\rm}

\newtheorem{remark}[theorem]{Remark}{\rm}

\newcounter{parentenv}

\newcommand{\R}{{\mathbb R}}

\newcommand{\N}{{\mathbb N}}

\newcommand{\s}{{\bm s}}

\begin{document}

\title[Differential systems with "maxima".]{Existence results for some  classes  of  differential systems with "maxima" }
\author[E. Mangino, E. Pascali]{Elisabetta  Mangino, Eduardo Pascali}
\date{}
\address{ Dipartimento di Matematica e Fisica ``Ennio De Giorgi'', Universit\`a del Salento, Via per Arnesano, I-73100 LECCE, Italy}
\email{elisabetta.mangino@unisalento.it}
\email{eduardo.pascali@unisalento.it}


\begin{abstract}
Local existence properties of initial boundary value problems associated with a new type of systems of differential equations with ``maxima"  are investigated. \end{abstract}

\maketitle

\begin{center}
\emph{ \small dedicato al Prof. Antonio Avantaggiati per il suo novantesimo compleanno}
\end{center}

\section{Introduction}
In this short note we consider a class of functional differential equations (and systems) that can be used to describe complex evolutionary  phenomena in which the future behaviour  depends not only on the present state but also  on the past history. The model problem is an initial value problem (IVP)  associated with   a modified logistic equation which contains the maximum of the square of the unknown function over a past interval: 
\begin{equation}\label{P1}
\begin{cases}\dot{x}(t)= x(t)- \max_{[0,t]} x^2(s) \quad t\geq 0;\\
  x(0)= x_0\end{cases}
\end{equation}
where $x_0\in\R$.

As it is emphasized in the book \cite{DAN}, the application of the classical logistic equation in the setting of experimental sciences entails two order of difficulties: on one hand the necessity of experimentally setting some of the parameters appearing in the equation, and on the other hand the fact that the derivative changes sign exactly when  a certain value of the function is reached. 
To tackle with the second problem, often an  apriori set delay $\tau$ is considered in the equation. It is evident that there are situations in which neither the delay nor the parameters can be determined on an experimental base. The problem \eqref{P1} seems to be more appropriate to deal with those cases.

Analizing \eqref{P1},  it is obvious that, if $x_0=0$  (resp. $x_0=1$),  then  the constant function $x\equiv 0$ (resp. $x\equiv1$)  is a solution.
Moreover, if $x\in C^1([0,T])$ is a solution of \eqref{P1}, we observe that:
\begin{itemize}
\item if $x_0 <0$ or $x_0>1$,  then $\dot{x}(0)<0$. Therefore,  in a neighbourhood of $0$, $\dot{x}(t)<0$  and  the equation reduces to $\dot{x}(t) = x(t) - x_0^2.$
\item if $0< x_0 <1,$ then $\dot{x}(0)>0.$ Therefore,   in a neighbourhood of $0$, $\dot{x}(t)>0$ and the equation reduces  to the well know equation $\dot{x}(t) = x(t) -  x^2(t).$
\end{itemize} 

%The usual theorems for the existence of solution seem not applicable directly at the previous IVP. 

These easy considerations show that the problem \eqref{P1}    somehow "contains" two different types of problems, on the basis of the initial value.

Moreover the IVP \eqref{P1} features also the following  strange behaviour.
Let $t_0 >0$ and assume that $0< x_1 <1$: then a solution of the following IVP 
\begin{equation}\label{PP1}
\dot{x}(t)= x(t)- \max_{[0,t]} x^2(s) \quad t_0 \leq t; \quad x(t_0)= x_1
\end{equation}
could be an extension of a solution either of the IVP
\begin{equation}\label{PPP1}
\dot{x}(t)= x(t)- \max_{[0,t]} x^2(s) \quad 0 \leq t; \quad x(0)= y_0
\end{equation}
or of the IVP
\begin{equation}\label{PPPP1}
\dot{x}(t)= x(t)- \max_{[0,t]} x^2(s) \quad 0 \leq t; \quad x(0)= z_0
\end{equation}
for suitable $0<y_0<1, \quad 1<z_0.$
This "uncertainty" situation for a solution $x=x(t)$ could appear  at all time $t>0$ for which $0< x(t) <1.$

More generally, we are going to consider the system
\begin{equation}\label{P2}\begin{cases} \dot{x}(t)=\displaystyle{f\Big(t, x(t), \max_{s\in [0,t]}g_1(x_1(s)),  \dots, \max_{s\in [0,t]}g_m(x_m(s))\Big)}, \qquad t\geq 0\\
x(0)=x_0\end{cases}
\end{equation}
where $x_0\in\R^m$, 
 $f\in C([0,+\infty[ \times \R^{2m}, \R^{m})$ and is locally Lipschitz with respect  to the second variable and the functions  $g_i\in C(\R)$ are  locally Lipschitz  on $\R$, for every $i=1, \dots, m$.

This type of systems  belongs to the class of  systems of differential equations with "maxima". We refer to the monograph \cite{BH} for a survey of motivations and techniques on the subject.  In particular,  Section 3.3 of \cite{BH} is devoted to the study of IVP associated with scalar differential equations of  the type
\begin{equation}\label{P3} \begin{cases} \dot{x}(t)=f(t, x(t), \max_{s\in [0,t]} x(s)) \qquad t\geq 0\\
x(0)=x_0\end{cases}
\end{equation}
Clearly, even in the scalar case, the class of problems \eqref{P2} is wider than \eqref{P3}. 

Our aim is to provide first, via fixed point theory, a local esistence result for the general system \eqref{P2}.
Afterwards, in particular situations as \eqref{P1}  and for systems of the type 
$$
\begin{cases}
\dot{x}(t)= x(t)- \max_{s\in [0,t]} y(s)\\
\dot{y}(t)= y(t)- \max_{s\in [0,t]} x(s)\end{cases} \qquad \mbox{or}\qquad 
\begin{cases}
\dot{x}(t)= x(t)- \max_{s\in [0,t]} y^2(s)\\
\dot{y}(t)= y(t)- \max_{s\in [0,t]} x^2(s)\end{cases},
$$
we will provide more precise existence results by the use of Peano-Picard's approximation.

%The same considerations for the system of ordinary differential equation given in the following are no easy to give.
\section{Local existence results via contraction theorem}

We start with two remarks that will help along the proofs of our results.

\begin{remark}{\rm Let $g, h\in C([a, b])$. Then
\[ |\max_{[a,b]} g - \max_{[a,b]} h| \leq \max_{[a,b]}|g-h|.\]
Indeed, assume that  $\max_{[a,b]} g \geq  \max_{[a,b]} h$ and let $x_0\in [a,b]$ such that $\max_{[a,b]}g=g(x_0)$. Then, 
\[ |\max_{[a,b]} g - \max_{[a,b]} h| = g(x_0) - \max_{[a,b]} h \leq g(x_0)-h(x_0)= |g(x_0)-h(x_0)| \leq \max_{[a,b]}|h-g|.\]}
\end{remark}

\begin{remark} {\rm Let $g\in C([a,b])$. Then the function 
$$h(s)=\max_{\tau\in [0,s]} g(\tau), \quad s\in [a,b]$$
is continuous. Indeed let $s_0\in [a,b]$. Fix $\varepsilon>0$ and consider $\delta>0$ such that $|g(\tau)-g(s)|>\varepsilon$ if $|\tau-s|<\delta$. 
For any $s_0<s< s_0+\delta$, it can happen that $h(s)=h(s_0)$ or that $h(s)=g(\overline\tau)$ for some $\overline\tau\in[s_0,s]$. In the first case obviously $h(s)-h(s_0)<\varepsilon$, while in the second case
\[ |h(s)-h(s_0)|=h(s)-h(s_0) \leq g(\overline\tau) -g(\s_0)<\varepsilon.\]
Therefore $\lim_{s\to s_0^+}h(s)=h(s_0)$. 

If $s_0-\delta<s<s_0$, then $h(s_0)=h(s)$ or $h(s_0)=g(\overline\tau)$ for some $\overline\tau\in[s,s_0]$.
In the last case,
\[ |h(s)-h(s_0)|=h(s_0)-h(s) \leq g(\overline\tau) -g(s)<\varepsilon.\]
So we get that $\lim_{s\to s_0^-}h(s)=h(s_0)$.
} \end{remark}

\begin{theorem}   
Let  $x_0\in\R^m$, 
 $f\in C([0,+\infty[ \times \R^{2m}, \R^{m})$ and  locally Lipschitz with respect  to the second variable and  $g_i\in C(\R)$ locally Lipschitz  on $\R$, for every $i=1, \dots, m$.

Given  $\alpha>0$ and $T>0$, 
set  
\[ M_{\alpha, T}:=\max \left\{ ||f(t,u,v)||\, \mid\, t\in [0,T], u\in [x_0-\alpha, x_0+\alpha]^m, v\in \prod_{i=1}^m g_i([x_0-\alpha, x_0+\alpha])\right\}\]
and assume $M_{\alpha,T}>0$.
Let $L_{\alpha, T}>0$ and $L_\alpha>0$  be such that for every $t\in [0,T]$, $u_1, u_2 \in [x_0-\alpha, x_0+\alpha]^m$, $v_1, v_2\in \prod_{i=1}^m g_i([x_0-\alpha, x_0+\alpha])$ 
and for every $x,y \in [x_0-\alpha, x_0+\alpha]$
\begin{align*} &||f(t, u_1, v_1)-f(t, u_2, v_2)||\leq L_{\alpha, T}(||u_1-v_1||+||u_2-v_2||)\\
&|g_i(x)-g_i(y)|\leq L_{\alpha} |x-y|.
\end{align*}
Then, for every 
$$0<\overline T<\min\left\{ \frac{\alpha}{M_{\alpha, T}}, \frac{1}{L_{\alpha, T}(1+L_\alpha\sqrt{m)}}, T\right\}$$
there exists $x\in C^1([0, \overline T]; \R^m)$ solution of the IVP \eqref{P2}.
\end{theorem}

\begin{proof} We will apply the Banach Fixed Point Theorem. 

Indeed, observe first that the existence of a $C^1$ solution of problem \eqref{P2} is equivalent to the existence of a continuous solution of the integral problem
\begin{equation}\label{intp}
x(t)=x_0+ \int_0^t \displaystyle{f\Big(s, x(s), \max_{\tau\in [0,s]}g_1(x_1(\tau)),  \dots, \max_{\tau\in [0,s]}g_m(x_m(\tau))\Big)} ds.\end{equation}
Fix 
$$0<\overline T<\min\left\{ \frac{\alpha}{M_{\alpha, T}}, \frac{1}{L_{\alpha, T}L_\alpha\sqrt{m}}, T\right\}$$ 
and 
consider the map $F:C([0, \overline T]; \R^m)\rightarrow C([0, \overline T]; \R^m)$ defined by 
$$F(x)(t)=x_0+ \int_0^t \displaystyle{f\Big(s, x(s), \max_{\tau\in [0,s]}g_1(x_1(\tau)),  \dots, \max_{\tau\in [0,s]}g_m(x_m(\tau))\Big)} ds$$
and the ball
\[ X:=\Big\{ x\in C([0, \overline T]; \R^m)\, \mid\, ||x(t)-x_0||\leq \alpha\ \ \forall t \in [0, \overline T]\Big\}.\]
Clearly $X$ is a complete metric space, with respect the the distance induced by the norm of $C([0, \overline T]; \R^m)$:
$$||x||_\infty:= \sup_{t\in[0,\overline T]} ||x(t)||, \qquad x\in C([0, \overline T]; \R^m).$$
If $x\in X$, then
\begin{align*}& ||F(x)-x_0||_\infty \leq \\ 
\leq & \sup_{0\leq t\leq \overline T} \int_0^t \left\Vert \displaystyle{f\Big(s, x(s), \max_{\tau\in [0,s]}g_1(x_1(\tau)),  \dots, \max_{\tau\in [0,s]}g_m(x_m(\tau))\Big)} \right\Vert ds\\
\leq & \overline T M_{\alpha, T} \leq \alpha.
\end{align*}
Hence $F(X)\subseteq X$. On the other hand, for every $x, y\in X$, it holds
\begin{align*}
&||F(x)- F(y)||_\infty \leq\\
&\leq  \overline T L_{\alpha, T} (||x-y||_\infty +\\
&+ \max_{s\in [0,\overline T]}||\Big(\max_{\tau\in [0,s]}g_1(x_1(\tau)),  \dots, \max_{\tau\in [0,s]}g_m(x_m(\tau))\Big) - \Big(\max_{\tau\in [0,s]}g_1(y_1(\tau)),  \dots, \max_{\tau\in [0,s]}g_m(y_m(\tau))\Big)||\leq\\
&\leq  \overline T L_{\alpha, T} \Big(||x-y||_\infty  + \sqrt{m}\max_{s\in [0,\overline T]}\max_{i=1}^m  | \max_{\tau\in [0,s]}g_i(x_i(\tau))) - \max_{\tau\in [0,s]}g_i(y_i(\tau)))|\Big)\leq\\
& \leq \overline T L_{\alpha, T}\Big(||x-y||_\infty+ \sqrt{m} \max_{i=1}^m \max_{\tau\in [0,\overline T]}|g_i(x_i(\tau))- g_i(y_i(\tau))|\Big)\leq\\
& \leq \overline T L_{\alpha,T} (1+\sqrt m L_\alpha)||x-y||_\infty.
\end{align*}
Therefore $F$ is a contraction on $X$ and it has a unique fixed point.
\end{proof}

\begin{remark}
{\rm The previous result applies, for example, to the following types of problems
$$
\dot{x}(t)= \alpha(t)x(t)- \beta(t)\max_{s\in [0,t]} x^2(s) \quad 0 \leq t; \quad x(0)= x_0;
$$
$$
\dot{x}(t)= \alpha(t) x(t)- \beta(t) \max_{s\in [0,t]} x(s) \quad 0 \leq t; \quad x(0)= x_0;
$$
$$
\dot{x}(t)= \alpha(t) x(t)- \beta(t) \max_{s\in [0,t]} |x(s)| \quad 0 \leq t; \quad x(0)= x_0.
$$
under  suitable conditions on the function $\alpha, \beta$.}
\end{remark}

\section{Existence proofs with approximations}

\begin{theorem}\label{T1}
Consider the following problem
\begin{equation}\label{eq:P2}
\begin{cases}
\dot{x}(t)= x(t)- \max_{s\in [0,t]} x^2(s) \quad  t\geq 0; \\
 x(0)= x_0\end{cases}
\end{equation}
with $x_0\not=1$. Let $\alpha>1$ and 
\[0<T^* <  \frac{\alpha -1}{\alpha (1+\alpha|x_0|)}.\]
Then there  exists a solution  $x\in C^1([0,T^*])$ of \eqref{eq:P2}.\end{theorem}

\begin{proof} We prove the existence of a solution via Peano-Picard's approximations.
Set $x_0 (t) = x_0$  for every $t\geq 0$ and define
$$
x_n(t) = x_0 + \int_0^t x_{n-1}(s)ds -\int_0^t \max_{[0,s]}x_{n-1}^2(\eta) d\s \quad t\in [0,T^*], \qquad n\geq 1
$$
It immediate  to prove that 
\begin{equation}\label{EQ1}
x_n(t) = x_0 g_n(t)\quad \forall n \in N, t\geq 0
\end{equation}
where $g_0\equiv 1$ and 
\[ g_n(t)= [1+\int_0^t g_{n-1}(s)ds -x_0\int_0^t \max_{[0,s]}[g_{n-1}(\eta)]^2ds].\]

By induction, using the choice of $T^*$,  we easily get that 
\[ \forall n\in\N, t\in [0,T^*]\ \ \ |g_n(t)|\leq \alpha,\]
%Indeed $g_0\leq\alpha$ and for every $n\in\N$, $t\in [0,T^*]$, by the inductive step: 
%\[|g_{n}(t)|\leq 1 +\alpha T^* + |x_0| \alpha^2 T^*\leq \alpha.\[
and, as a consequence, that 
\[ |g_{n+1}(t)-g_n(t)|\leq \frac{|1-x_0|}{1+2\alpha|x_0|} \frac{(1+2\alpha|x_0|)^{n+1}t^{n+1}}{(n+1)!}.\]
Then the sequence $(g_n)_n$ is uniformly convergent on $[0, T^*]$ and therefore also the sequence $(x_n)_n$ is uniformly convergent om $[0,\overline T]$. It is immediate that its uniform limit is a solution of the problem \eqref{eq:P2}.
\end{proof}

\begin{remark}{\rm It is worth noticing that  
\[ T^*<\max_{\alpha\geq 1}\frac{\alpha -1}{\alpha (1+\alpha|x_0|)}.\]}
\end{remark}

%XXXXXXXXXXXXXXXXXXXXXXXXXXXXXXXXXXXXXXXXXXX
%XXXXXXXXXXXXXXXXXXXXXXXXXXXXXXXXXXXXXXXXXX
\bigskip

We consider  now the following system
\begin{equation}\label{eq:P3}
\begin{cases}
\dot{x}(t)= x(t)- \max_{s\in [0,t]} y(s)
\\
\dot{y}(t)= y(t)- \max_{s\in [0,t]} x(s)
\\
x(0)=x_0 \quad y(0)=y_0\end{cases}
\end{equation}
with $x_0, y_0\in\R$. 
We remark that \eqref{eq:P3} is equivalent to the functional system 
\begin{equation}\label{eq:P4}\begin{cases}
x(t) = x_0 +\int_0^t x(s)ds-\int_0^t \max_{[0,s]}y(\tau)ds,
\\
y(t) = y_0 +\int_0^t y(s)ds-\int_0^t \max_{[0,s]}x(\tau)ds,\end{cases}
\end{equation}
The following theorem holds.
\begin{theorem}\label{T2}
Assume that $x_0 >0$, $y_0>0$ and $x_0\neq y_0$.
Then  for all $T>0$ there exists a solution  $(x(t), y(t))\in C^1([0,T])^2$ of the system \eqref{eq:P3}.
\end{theorem}

\begin{proof} Assume $0< y_0 <x_0$ and consider the sequences of functions $(x_n)$ and $(y_n)$ defined on $[0,+\infty[$ by 
\begin{align*}
&x_0 (t)=x_0 \quad y_0(t)=y_0\\
&x_{n+1}(t) = x_0 +\int_0^t (x_n(s)-y_0)ds\\
&y_{n+1}(t) = y_0 +\int_0^t (y_n(s)-x_n(s))ds
\end{align*}

It holds that, for every $n\in\N$ and for every $t\geq 0$, $x_n(t)\geq y_0$ and $y_n(t)\leq x_n(t)$.
Indeed, the assertion is obviously true if $n=0$. Assuming that $x_n(t)\geq y_0$ and $y_n(t)\leq x_n(t)$ for every $t\geq 0$, we get that
\begin{align*} &x_{n+1}(t)-y_0= x_0-y_0 + \int_0^t (x_n(s)-y_0)ds \geq 0, \\
 &x_{n+1}(t)-y_{n+1}(t) =x_0-y_0 + \int_0^t (x_n(s)-y_n(s))ds \leq 0.\end{align*}
As a consequence we get that, for every $n\in\N$, $\dot{x}_n\geq 0$ and $\dot{y}_n\leq 0$ and consequently
\[ \max_{[0,s]}x(\tau)=x(s), \ \  \max_{[0,s]}y_n(\tau)=y_n(0)=y_0.\]
Therefore, for  the sequences $(x_n)$ and $(y_n)$, it holds that 
\begin{align*}
x_{n+1}(t) = x_0 +\int_0^t x_n(s)ds-\int_0^t \max_{[0,s]}y_n(\tau)ds,
\\
y_{n+1}(t) = y_0 +\int_0^t y_n(s)ds-\int_0^t \max_{[0,s]}x_n(\tau)ds,
\end{align*}

By induction, one can prove that   for every $n\in\N$ and every $t\geq 0$
\begin{align*}
 & |x_{n+1}(t)-x_n(t)|\leq  |x_0-y_0| \frac{t^{n+1}}{(n+1)!}\\
&|y_{n+1}(t)-y_n(t)| \leq |x_0-y_0|T \frac{t^{n+1}}{n!}
\end{align*}
Hence the   sequences $(x_n)$ and  $(y_n)$ are uniformly convergent on $[0,T]$ to  continuous functions $x_\infty=x_\infty(t)$ and $y_{\infty}=y_{\infty}(t)$ and 
 the couple  $(x_{\infty},y_{\infty})$  is a solution of the functional system \eqref{eq:P4}.
\end{proof}

\begin{remark}{\rm It is worth observing  that the proof fails if $x_0=y_0.$ Moreover the proof highlights the difference with the system 
$$
\dot{x}(t)= x(t)- y(t)
$$
$$
\dot{y}(t)= y(t)- x(t).
$$}
\end{remark}

\begin{remark}  {\rm More interesting seems to be the study of the following general system
$$\begin{cases}
\dot{x}(t)= a(t)x(t)- b(t)\max_{s\in [0,t]} y(s)\\
\dot{y}(t)= c(t)y(t)- d(t)\max_{s\in [0,t]} x(s)\\
x(0)=x_0 >0, \quad y(0)=y_0>0 \end{cases}
$$
where the functions $a,b,c,d$ are continuous, non negative and defined on the interval $[0,T]$.

If  the functions $a,b,c,d$ are constant, one can prove the following partial results.

If $A=ax_0-by_0 <0 \quad B= cy_0 -dx_0<0$ and $a>0, \quad c>0,$ then a solution is the following couple of functions
$$
x(t)=x_0 +A\frac{1}{a}[e^{at}-1] \quad y(t)=y_0+B\frac{1}{c}[e^{ct}-1].
$$
and therefore more information follow. For example we have that 
\begin{align*}
&x(t)=0 \Leftrightarrow t= \frac{1}{a}\log \frac{by_0}{|A|}
&y(t)=0\Leftrightarrow t=\frac{1}{c}\log \frac{dx_0}{|B|}.
\end{align*}
For different situations, such as $A>0, B<0$, or $A<0,B>0$,  or $A>0,B>0$ an explicit representation for the solution is not available.}\end{remark}

%XXXXXXXXXXXXXXXXXXXXXXXXXXXXXXXXXXXX
%XXXXXXXXXXXXXXXXXXXXXXXXXXXXXXXXXXXX
Next we consider the following problem, for $t\geq 0$

\begin{equation}\label{eq:P5}\begin{cases}
\dot{x}(t)= x(t)- \max_{s\in [0,t]} y^2(s)\\
\dot{y}(t)= y(t)- \max_{s\in [0,t]} x^2(s)\\
x(0)=x_0>0, \\
y(0)=y_0>0.\end{cases}\end{equation}

\begin{theorem} If $T, c_0>0$ satisfy
\begin{align*}
&|x_0|+|x_0-y_0^2|T \leq c_0, \quad |y_0|+|y_0-x_0^2|T \leq c_0;\\
&|x_0|+c_0T+c_0^2T \leq c_0, \quad |y_0|+c_0T+c_0^2T \leq c_0.
\end{align*}
 then there exists $(x,y)\in C^1([0,  T]; \R^2)$ solution of the IVP \eqref{eq:P5}
 \end{theorem}

\begin{proof} 
The initial problem \eqref{eq:P5} is equivalent to the  following functional system.

\[ \begin{cases}
x(t) =x_0 + \int_0^t x(s)ds - \int_0^t \max_{[0,s]}y^2(s)ds;
\\
y(t) =y_0 + \int_0^t y(s)ds - \int_0^t \max_{[0,s]}x^2(s)ds.
\end{cases}\]

As usual, we  define the  sequences of functions
$(x_n)$ and  $(y_n)$   on $[0,+\infty[$ by:
\begin{align*}
&x_0(t)=x_0, \ \ y_0(t)=y_0\\
&x_{n+1}= x_0 + \int_0^t x_n(s)ds - \int_0^t \max_{[0,s]}y_n^2(s)ds\\
&y_{n+1}= x_0 + \int_0^t x_n(s)ds - \int_0^t \max_{[0,s]}x_n^2(s)ds.\end{align*}

Under the assumptions, it is immediate to prove by induction that 
$$
|x_n(t)| \leq c_0, \quad |y_n(t)| \leq c_0 \quad \forall n \in N, \quad t\geq 0.
$$
Consequently
\begin{align*}
&|x_{n+1}(t) -x_n(t)| \leq \frac{c_0}{T} (1+2c_0)^n \frac{t^{n+1}}{(n+1)!};\\
& |y_{n+1}(t) -y_n(t)| \leq \frac{c_0}{T} (1+2c_0)^n \frac{t^{n+1}}{(n+1)!}.  
\end{align*}
Indeed, the last assertion is immediately true if $n=0$. Assuming it for $n$, we get that 
\begin{align*}
|x_{n+1}(t)-x_n(t)|\leq &\int_0^t |x_n(t)-x_{n-1}(t)|dt + \int_0^t \left\vert \max_{[0,s]}y_n^2(s) - max_{[0,s]}y_{n-1}^2(s)\right\vert ds \leq \\
&\frac{c_0}{T} (1+2c_0)^{n-1} \frac{t^{n+1}}{(n+1)!} + \int_0^t\max_{[0,s]}\left\vert y_n^2-y_{n-1}^2\right\vert ds\leq \\
&\frac{c_0}{T} (1+2c_0)^{n-1} \frac{t^{n+1}}{(n+1)!} +  2c_0 \int_0^t\max_{[0,s]}\left\vert y_n-y_{n-1}\right\vert ds\leq \\
& \frac{c_0}{T} (1+2c_0)^n \frac{t^{n+1}}{(n+1)!}.
\end{align*}

Hence the sequences $(x_n)$ and $(y_n)$ are uniformly convergent to continuous functions $x_\infty,  y_\infty$ defined in the interval $[0,T]$,
that solve the functional system.
\end{proof}

%We remark explicitly that the existence interval $[0,T]$ change with the different types of proof. 

%XXXXXXXXXXXXXXXXXXXXXXXXXXXXXXX
%XXXXXXXXXXXXXXXXXXXXXXXXXXXXXXX

\begin{remark}
{\rm The methods we have considered could also be applied to investigate 
a version of  Lokta-Volterra systems with "maxima", namely
$$
\dot{x}(t)=x(t) - \max_{s\in [0,t]}{x(t)y(t)}; \quad \quad \dot{y}(t)=y(t) + \max_{s\in [0,t]}{x(t)y(t)}. 
$$ 
or other analogous equations and systems. }\end{remark}

\end{document}